\documentclass{amsart}[12pt]

\usepackage{geometry} 

\def\Ree{\mathbb{R}}

\def\Nee{\mathbb{N}}

\usepackage{hyperref} 

\begin{document}

\title{$Q$-curvature flow for GJMS operators with non-trivial kernel }

\author{Ali Fardoun}
\address{Laboratoire de Math\'ematiques, UMR 6205 CNRS 
Universit\'e de Bretagne Occidentale 
6 Avenue Le Gorgeu, 29238 Brest Cedex 3   
France}
\email{Ali.Fardoun@univ-brest.fr}
\author{Rachid Regbaoui }
\address{Laboratoire de Math\'ematiques, UMR 6205 CNRS 
Universit\'e de Bretagne Occidentale 
6 Avenue Le Gorgeu, 29238 Brest Cedex 3   
France}
\email{Rachid.Regbaoui@univ-brest.fr}


\begin{abstract}
We investigate the prescribed $Q$-curvature flow for GJMS operators with non-trivial kernel on compact manifold of even dimension.  When the total $Q$-curvature is negative, we identify a conformally invariant condition on the nodal domains of functions in the kernel of the GJMS operator,  allowing us to  prove the global existence of the flow and its convergence  at infinity to a metric which is conformal to the initial one, and  having  a prescribed $Q$-curvature.  If the total $Q$-curvature is positive, we show that the flow blows up in finite time.
\end{abstract}

\maketitle

\section{Introduction and statement of the results.}

\medskip

In conformal geometry, the existence of conformally covariant operators and the study of their associated curvature invariants in order to understand some connections between analytic and geometric properties of such objects have become an attractive subject of research in the last decades. A model example is the Laplace Beltrami operator on a compact surface $(M,g_0)$. More precisely, if $g = e^{2u}g_0$ is a metric conformal to  $g_0$  ($u$ being a smooth function on $M$), then the Laplacian with respect to $g$ enjoys the well known covariance property
$$
\Delta_g = e^{-2u}\Delta_{0} \ , \eqno (1.1)
$$
where $\Delta_0$ is the Laplacian with respect to $g_0$ with sign convention chosen so that the non-zero eigenvalues are negative. This operator is closely related with the Gauss curvature. Namely, we have 
$$K_g= e^{-2u}\left(-\Delta_{0}u + K_{0}\right) ,  \eqno (1.2)$$
where $K_g$ and $K_0$ are  the Gauss curvature  of $g$ and $g_0$ respectively. So, the quantity $  \int_M K_{g} dV_{g}$, where $dV_g$ is the volume element of $g$,  is conformaly invariant and is related to the topology of $M$ by the Gauss-Bonnet formula
$$ \int_M K_{g} \ dV_{g} = 2 \pi \chi(M) \ , $$
where $\chi(M)$ is the Euler Characteristic of $M$.

\medskip

An important result which is related to equation (1.2) is the uniformization theorem which asserts that every compact surface admits a (conformal) metric with constant Gauss curvature.

\medskip

The conformal covariance property (1.1) of the Laplacian operator in dimension two  no longer holds in higher dimensions, and one is led to ask if there exists a differential operator with the same behavior as the Laplacian in dimension two,  when dealing with conformal changes of the metric. In 1983, S.M.Paneitz \cite{sP} introduced a fourth order differential operator on 4-manifolds having the covariance property.  More precisely, if $(M, g_0)$ is a Riemannian manifold of dimension 4,  then the Paneitz operator  $P_0$ is defined by
$$
P_0v := \Delta_0^2 v - \hbox{div}_0 \left( \Bigl({2\over 3}R_0g_0 - 2 \hbox{Ric}_0 \Bigr)\cdot dv\right)
 , \quad v \in C^{\infty}(M) \,,
$$
where $\hbox{Ric}_0$ is the Ricci tensor of $g_0$,  $R_0$ its scalar curvature and $\hbox{div}_0$
is the divergence operator with respect to $g_0$. Under a conformal change of the metric $g = e^{2u}g_0$, the Paneitz operator $P_g$ with respect to $g$ satisfies the covariant property
$$
P_g = e^{-4u}P_0  \,.
$$
A natural conformal invariant introduced  by T.P.Branson-B.Oersted \cite{tB},  is the so-called $Q$-curvature
$$
Q_0 = -{1\over 6}\left( \Delta_0R_0 -R_0^2 + 3|\hbox{Ric}_0|^2\right) \,.
$$
From the Chern-Gauss-Bonnet formula
$$
\int_M Q_0 \ dV_0 + {1\over 4} \int_M |W_0|^2 \ dV_0 = 8\pi^2 \chi(M) \,,
$$
where $W_0$ is the Weyl tensor of $g_0$ ; it follows that the total $Q$-curvature
$$
k_p:= \int_{M}Q_0 \ dV_0
$$
is conformally invariant since the Weyl tensor is pointwise conformally invariant. The $Q$-curvature is intimately related to the Paneitz operator in the sense that if one performs a conformal change of  the metric : $g = e^{2u}g_0$, then the $Q$-curvature of $g$ is given by
$$
Q_g = e^{-4u}\left(P_0 u + Q_0\right)\,,  \eqno (1.3)
$$
which also asserts (by integrating (1.3) on $M$) that the quantity $k_p$ is conformally invariant.

\medskip

For general dimension $n$, C.R.Graham {\it et al} \cite{cG} proved the existence of conformally covariant operators generally referred to as the Graham-Jenne-Mason-Sparling operators (GJMS operators, for short).  Furthermore, when $n$ is even, among these operators, there is one of order $n$ which we will still denote by $P_0$\, . When $n = 2$, the operator $P_0$ is the Laplacian and when $n=4$ the operator $P_0$ is the Paneitz operator. This operator shares the same properties as the Paneitz operator in dimension $4$, namely  $P_0$  is a self-adjoint differential operator, which annihilates the constant functions with leading term $(\Delta_0)^{n/2}$  and which has the following behavior under a conformal change of the metric;  if we set $g = e^{2u}g_0$, then the GJMS operators  $P_g$ with respect to $g$  satisfies 
$$
P_g = e^{-nu} P_0  \,.  \eqno (1.4)
$$
When $ M = \Ree^n$, the operator $P_0$ is given by $(\Delta_0)^{n/2}$.
When $ n = 2$, $P_0$ has trivial kernel on closed manifolds. But this is no longer true in higher dimensions, M.Eastwood-M.A.Singer \cite{mE1} give explicit examples in dimension 4. Note also that by formula (1.4), the kernel of $P_0$ is invariant under conformal change of metric.

\medskip

As in dimension 4, the GJMS operator $P_0$ is related to a curvature quantity, the so-called $Q$-curvature, which transforms according to
$$
Q_g = e^{-nu}\left( P_0 u + Q_0\right)  \,,     \quad (g = e^{2u}g_0) \,.      \eqno (1.5)
$$
And by integrating (1.5) on $M$, we see that the total $Q$-curvature $k_p =  \int_{M}Q_0 \ dV_0 $ is conformally invariant.
The $Q$-curvature and Paneitz operator and their high-dimensional analogues are important objects of study in conformal geometry and in mathematical physics.

\medskip

As for the uniformization theorem for surfaces, one can ask whether a compact n dimensional manifold (n even) admits conformal metrics with constant $Q$-curvature.  These metrics are critical points of the functional
$$
 II(u) := {n \over 2} \int_{M} P_0 u\cdot u \ dV_0 + n \int_M Q_0 u \ dV_0 - k_p \log\left(\int_M e^{nu} dV_0\right).
 $$
  The  difficulty in the study of this functional is that in general it is not coercive, furthermore, it can be  unbounded from below and from above. This is due to the possibility of negative eigenvalues of the operator $P_0$ and to the large values that the conformal invariant $k_p$ may have.
In dimension 4, S.Y.A.Chang - P.Yang \cite{aC} first studied the uniformization for the $Q$-curvature; by minimizing the functional $II$, they constructed conformal metrics of constant $Q$-curvature, under the hypothesis that the operator $P_0$ is positive with  trivial kernel (i.e,  consisting of constant functions) and $k_p < 16\pi^2$. Under the hypothesis that  $P_0$ is positive with trivial  kernel, the following Adams \cite{dA} inequality holds
for all $u \in H^{2}(M)$
$$
\log \left (\int_M e^{4(u- \overline{u})} \ dV_0 \right) \leq  C + {1
\over 8 \pi^2}  \int_M P_0u \cdot u \ dV_0 \,.
\eqno (1.6)
$$
where $C$ is a constant depending only on $M$ and where the average of $u$, $ \overline{u} = {1 \over \hbox{vol}_{g_0} (M) } \int_M u \ dV_0$, is the projection of $u$ into the kernel of $P_0$. By Adams inequality (1.6), the functional $II$ is bounded from below, coercive and its critical points can be found as global minima. M.J.Gursky \cite{mG} showed that these conditions are satisfied if the Yamabe invariant of $(M, g_0)$ is positive, $k_p \ge 0$ and $M$ is not conformally equivalent to $S^4$. This uniformization result has been extending by Z.Djadli - A.Malchioldi \cite{zD} when $P_0$ is not necessarily positive under the condition that the kernel of $P_0$ is trivial  (consists only of constants), and  $k_p \not = 16 k \pi^2, \  k \in \Nee^*$. Later,  C.B.Ndiaye \cite{cN} generalized the result of Z.Djadli - A.Malchioldi \cite{zD}  to  even higher dimension n   assuming  $P_0$ with trivial kernel and $ k_p \not = k (n-1)!  \omega_n, k \in \Nee^*$,  where $\omega_n$ is the volume of the standard sphere $\mathbb{S}^n$. He used in this case the analogue of Adams inequality (1.6) in higher dimension : 
$$
\log \left (\int_M e^{n(u- \overline{u})} \ dV_0 \right) \leq  C + {n
\over 2(n-1)! \omega_n }  \int_M P_0u \cdot u \ dV_0 \,.
\eqno (1.7)
$$

\medskip

More generally, the prescribed $Q$-curvature problem is : given a  function $f$ on $M$, does there   exist a conformal metric (to the initial metric) whose $Q$-curvature equals to $f$ ? This is equivalent to solve the following PDE
$$
f = e^{-nu}\left( P_0 u + Q_0\right)  \,,     \quad (g = e^{2u}g_0) \,.      \eqno (1.8)
$$
There have been many results recently in the study of equation (1.8) (see for example \cite{pB1},  \cite{pB2}, \cite{sB2}, \cite{pD}, \cite{fR}, \cite{jW}). It is easy to check that solutions of (1.8) are critical points of the following functional on $H^{n\over 2}(M)$ :
$$II_f (u):= {n \over 2} \int_{M}  P_0u \cdot u \ dV_0 + n \int_M Q_0 u \ dV_0 - k_p \log\left(\int_M f e^{nu} dV_0\right) . \eqno (1.9) $$

\medskip

Using a similar method to that used by R.Ye \cite{rY} to prove Yamabe's theorem
for locally conformally flat manifolds, S.Brendle \cite{sB1} considered  the evolution of the metric $g$ under the 
flow:
$$
\begin{cases} \partial_t g  = - \left( Q_g - { k_p \over \int_M f \  dV_g  }f  \right) g \cr \cr
 g(0)= e^{2u_0} g_0,  \    u_0 \in C^{\infty} (M) \end{cases}  \eqno (1.10)
$$ 
where $dV_g$ is the volume element associated to the metric $g(t)$, and  we suppose that $f >0$. Since equation (1.10) preserves the conformal structure of $M$, then \ $ g(t) = e^{2u(t)} g_0$ , where $u(t) \in C^{\infty}(M)$ with initial condition $u(0) = u_0$ . Then the flow (1.10) takes the form

 $$
\begin{cases} \partial_t u  = -{1 \over 2} e^{-nu} P_0u -{1 \over 2} e^{-nu} Q_0 + {1 \over 2}\frac{ k_p }{\int_M f e^{nu} \ dV_0 }f , \cr u(0)= u_0.   \end{cases}
\eqno (1.11)
$$
Since equation (1.11) is parabolic, by classical methods it admits a solution $ u \in C^{\infty} ([0,T) \times M)$ where $T \leq + \infty$ denotes the maximal time of existence. When $T < + \infty$, we say that the flow (1.10)-(1.11) blows up in  finite time. By integrating (1.11) over $M$ with respect to the volume element of $g(t)$, we see that the volume of $M$ with respect to $g(t)$ remains constant, that is, 
$$\int_M e^{nu(t)} dV_0 =  \int_M e^{nu_0} dV_0 \  \   \forall t \in [0, T). \eqno (1.12)$$
One can also check that the functional $ II_f$ defined by (1.9) above is decreasing along the flow. More precisely, we have 
$${d\over d t}II_f(u(t))  = -2n \int_M e^{nu(t)} |\partial_t u(t)|^2 dV_0 \le 0  \ \ \forall t \in [0, T).   \eqno (1.13) $$ 

\medskip

\noindent S.Brendle \cite{sB1} proved that if the GJMS operator $P_0$ is positive with  trivial kernel
and $k_p< (n-1)! \omega_n$, where $\omega_n$ is the volume $\mathbb{S}^n$, then (1.10)-(1.11)  has
a solution which is defined for all time
and converges at infinity to a metric  which is conformal to $g_0$  and  whose  $Q$-curvature is  a constant times $f$. For a similar flow on the Euclidean sphere $\mathbb{S}^4$ (where $k_p =  (n-1)! \omega_n$ on $\mathbb{S}^n$)  see the work of  A.Malchiodi and M. Struwe \cite{aM}.  

\medskip

In this paper, we investigate the flow  (1.10)-(1.11) when the operator $P_0$ has non-trivial kernel. We are interested in the global existence and the convergence of such a flow.
Before stating our results, let us introduce  some notations which will be  be useful in the sequel. We denote by $\mathcal{N}(P_0)$ the kernel of $P_0$, that is 
$$\mathcal{N}(P_0) := \{ u \in C^{\infty}(M) \ : \  P_0 u = 0 \} . $$
We say that $P_0$ has a non-trivial kernel if $\mathcal{N}(P_0)$ has dimension at least two. 
As in \cite{rG}, we shall consider  a useful sub-space of $\mathcal { N}( P_0)$ as follows. We consider the linear conformally invariant form $ \mathcal { Q} : \mathcal { N}$$( P_0)
\to \Ree$ given by
$$
u \mapsto \mathcal{Q}(u) :=  \int_M Q_g u \ dV_g. $$
Denote by $ \mathcal { N(Q)}$ its kernel, that is, 
$$  \mathcal { N(Q)} := \left\{ \ u \in  \mathcal{ N}(P_0) \ : \  \int_M Q_0 u \ dV_0 = 0  \right\} .$$
If $k_p \not = 0$, then we have the following decomposition : 
$${\mathcal  N}( P_0) = \Ree \oplus  \mathcal{ N(Q)}. \eqno (1.14) $$
We note here that the decomposition (1.14) is not necessarily orthogonal.

\medskip

Adams inequality (1.7) is valid if $P_0$ has  trivial kernel, but  one can verify without difficulty that Adams inequality is still valid if $P_0$ has non-trivial kernel when  applied to functions whose orthogonal projection on the kernel of $P_0$ vanishes.  More precisely,  for $v \in H^{n\over 2}(M)$, if we denote by $\widehat v$ its  orthogonal projection 
onto the kernel of $P_0$, then we have 
$$
\log \left (\int_M e^{n(v- \widehat{v})} \ dV_0 \right) \leq  C + C  \int_M P_0v \cdot v \ dV_0 \,, 
\eqno (1.15) $$
where $C$ is positive constant depending only on $(M, g_0)$. 

\medskip

\medskip

For the statement of our main results we distinguish  two cases : $ k_p  > 0$ and $k_p < 0$. When $k_p > 0$, we prove that the flow blows up in finite time $T$ and we give an upper bound of the maximal time of existence $T$ (see Theorem 1.2 below). In the case that  $k_p < 0$, we prove under suitable conditions on the nodal domains of a basis of  $ \mathcal { N}(Q)$  that the flow exists for all time and converges at infinity to a metric conformal to the initial metric, and  having $Q$ curvature a multiple of the prescribed function $f$. More precisely, when $k_p < 0$, we suppose that $P_0$ is positive with non-trivial kernel of dimension $\nu + 1, \ \nu \ge 1$. Then $\mathcal { N(Q)}$ is of dimension $\nu$,  and let   $  \{\varphi_1,..,\varphi_{\nu} \}$   be a basis of $\mathcal { N(Q)}$.  We shall assume the following  on the nodal domains of the $\varphi_j$'s :
 
 \medskip
 
\noindent For any  $\varepsilon = ( \varepsilon_1 , ... , \varepsilon_{\nu} ) \in  \Ree^{\nu}$  with  $ \varepsilon_i = \pm  1, \ ( i = 1, ..., \nu )$,  there exists  $ x_{\varepsilon} \in M$  such that  

$$\Big(\hbox{sign}\left(\varphi_1(x_{\varepsilon})\right) , ... ,  \hbox{sign}\left(\varphi_{\nu}(x_{\varepsilon})\right)  \Big) = \varepsilon ,   \eqno (1.16) $$

\medskip

\noindent where  \  $\hbox{sign}(a) =  +1 $ \ if \ $a > 0$,  \  $\hbox{sign}(a) =  -1 $ \ if \ $a < 0$  \  and  $\hbox{sign}(a) = 0 $ \  if \ $a =0$. 

\medskip

One can easily check that  (1.16) is  automatically satisfied when  $\mathcal { N}$$( P_0)$ has dimension two and $Q_0 < 0 $.  Indeed, in this case, $\nu = 1$  and (1.16) means that $\varphi_{1}$ must change sign on $M$. But  since $\varphi_1 \in \mathcal { N(Q)}$, \ $\varphi_1$ must change sign if we assume that $Q_0<0$.  
 There are examples of manifolds such that the GJMS operator $P_0$ has kernel $\mathcal { N}$$( P_0)$ of dimension greater than 2 and such that $\mathcal { N(Q)}$ has a  basis  satisfying (1.16). For instance, consider the compact manifold $M = S^2 \times \Sigma_p$ where $\Sigma_p$ is a compact Riemann surface  of genus $p \geq 2$. Given to $S^2$ and $\Sigma_p$ their standard metrics of constant scalar curvature with equal magnitude but opposite sign. So $M$  has zero scalar curvature. The gradient on $M$ splits into an orthogonal sum $ \nabla = \nabla_1 + \nabla_2 $ where $\nabla_1$ (resp. $\nabla_2$) denote the gradient in the $S^2$ direction (resp. in the $\Sigma_p$ direction). Hence the Paneitz operator is equal to $P_0v =  \Delta_0^2 v - 2 \hbox{div}_0 \left( \nabla_1 v - \nabla_2 v \right)
 , \  v \in C^{\infty}(M),$
 and the $Q$-curvature is equal to $Q_0 = - 2$. First, a short computation shows that $P_0$ is a positive operator. Next, its kernel is of dimension 4 (see \cite{mE1}, \cite{mE2}), moreover if we denote by $\{ \varphi_1, \varphi_2, \varphi_3 \}$ the first spherical harmonics,  then $\{  \varphi_1, \varphi_2, \varphi_3 \}$ is an orthogonal basis of $\mathcal{N(Q)}$  satisfying (1.16).

\medskip

Our first main result is that,  under   hypothesis $(1.16)$,   for any initial value $u_0$ the flow (1.10)-(1.11) exists for all time and converges at infinity to a metric with  $Q$-curvature a multiple of the prescribed function $f$, namely :

\newtheorem{theo}{Theorem}[section]
\begin{theo}Let $(M,g_0)$ be a compact Riemannian manifold of even dimension $ n \geq 4$ such that the GJMS operator $P_0$ is positive with non-trivial kernel and $k_p < 0$. Suppose that  $\mathcal { N(Q)}$ has a  basis  satisfying (1.16), and let  $f$ be a smooth positive function on $M$.  Then for any $u_0 \in C^{\infty} (M)$, the evolution equation (1.10)-(1.11)  has a solution which is defined for all time and converges to a metric    $g_{\infty} = e^{2u_{\infty}}g_0$  whose  Q-curvature is \ 
 ${k_p \over \int_M f e^{2u_{\infty}}dV_0}  f  $. \end{theo}

\medskip

\noindent Theorem 1.1 generalizes the result obtained by S.Brendle \cite{sB1} when $ k_p < 0$ and  $P_0$ has  trivial kernel. 

\medskip

\newtheorem{rem}{Remark}[section]
\begin{rem} Since $\mathcal{N}(P_0)$ is conformally invariant, then hypothesis (1.16) in Theorem 1.1 is also conformally invariant.  
\end{rem}

\medskip

For the simplicity of the paper, in the case that $k_p > 0 $ we will suppose that the function $f$ is constant (otherwise some additional conditions on $f$ are needed). Our result in this case is as follows:

\begin{theo}Let $(M,g_0)$ be a compact Riemannian manifold of even dimension $n ( n \geq 4 )$ such that the GJMS operator $P_0$ has non-trivial kernel and $k_p > 0$. Let  $\varphi \in  \mathcal
{N(Q)}$,  then for all $u_0 \in C^{\infty} (M)$  satisfying  $ \int_M \varphi   e^{n u_0} \ dV_0 \not = 0$,  the solution $u$ of equation (1.11) blows up  in finite time $T$. Moreover, 
 $$T \le  {2 \over k_p}\int_M e^{n u_0} dV_0  \log\left({ \|\varphi\|_{L^{\infty}(M)} \int_M e^{n u_0} \ dV_0 \over \left| \int_M e^{n u_0} \varphi \ dV_0 \right| }\right) .$$ 
 \end{theo}

\bigskip

A particular case of Theorem 1.1 is the following corollary : 
\newtheorem{cor}{Corollary}[section]
\begin{cor} Let $(M,g_0)$ be a compact Riemannian manifold of even dimension $n$ ($ n \geq 4$) such that the GJMS operator $P_0$ has non-trivial kernel and $k_p > 0$. If there exists  $\varphi \in \mathcal {N( Q)}$ of constant sign then for all initial value $u_0 \in C^{\infty} (M)$ the flow (1.10)-(1.11)   blows up  in finite time.
\end{cor}

\medskip

\begin{rem} If $\varphi \in  \mathcal { N(Q)}$ then the set $ \left \{ u_0 \in C^{\infty} (M):  \int_M  \varphi e^{n u_0}  \ dV_0 \not = 0 \right \}  $ is dense in $C^{\infty} (M)$. This shows that when $k_p >0$, the flow (1.10)-(1.11)  blows up in finite time for \ `` almost all " intial data $u_0$. (See a proof of this fact in section 3 below).
\end{rem}

\medskip

\section{ The negative case.}

\medskip

In this section we prove Theorem 1.1.  First we need the following lemma. 

\medskip

\newtheorem{lem}{Lemma}[section]
 \begin{lem} Let $E = \left\{  (\varepsilon_1,..,\varepsilon_\nu) \in \Ree^{\nu}  :   \varepsilon_i  =  \pm1 \ \forall i = 1 , ... , \nu   \right \}$, and let   $\{\varphi_1,..,\varphi_{\nu} \} $  be a basis of $\mathcal{N(Q)}$ satisfying (1.16).  Then for any  $ \varepsilon = (\varepsilon_1, ..., \varepsilon_{\nu})  \in E$,  there exists a point $ x_{\varepsilon} \in M$ such that  for all $j = 1 , ..., \nu$ and all  $x$ in the geodesic ball $ B(x_{\varepsilon},r_0)$, we have 
 $$\varepsilon_j \varphi_j(x) \geq C ,$$
 where $r_0$ and $C$ are two positive constants depending only on $\{ \varphi_1, ..., \varphi_{\nu} \} $  (they do not depend  on  $\varepsilon$ and  $j$).\end{lem}

   \begin{proof} Let $ \varepsilon = (\varepsilon_1,..,\varepsilon_\nu )   \in E$,  then by $(1.16)$, there exists $x_{\varepsilon} \in M$ such that 
  $$ \Big(\hbox{sign}\left(\varphi_1(x_{\varepsilon})\right) , ... ,  \hbox{sign}\left(\varphi_{\nu}(x_{\varepsilon})\right)  \Big) = \varepsilon . $$

\noindent By continuity of the map $x \mapsto  (\varphi_1(x) , ..., \varphi_{\nu}(x)) $,  there exists   \ $r_{\varepsilon} > 0$ depending on $\varepsilon$ such that for all $x$ in the geodesic ball \ $\overline{B}(x_{\varepsilon}, r_{\varepsilon}) $, we have  
 $$\big( \hbox{sign}\left(\varphi_1(x)\right) , ... ,  \hbox{sign}\left(\varphi_{\nu}(x)\right)\big) = \varepsilon . \eqno (2.1)$$  
 Set $r_0 = \inf\{  \ r_{\varepsilon}  \ : \   \varepsilon \in E \  \} $. Since $E$ is a finite set, then $r_0 > 0$.  Now, let $C_{\varepsilon , j} = \inf\{ |\varphi_j(x)| \ : \ x \in \overline{B}(x_{\varepsilon}, r_0) \}$.  Since  (2.1) is satisfied for all $x \in \overline{B}(x_{\varepsilon}, r_{\varepsilon}) $, then   $C_{\varepsilon, j} > 0$. Hence  $\varepsilon_j \varphi_j(x) = |\varphi_j(x)| \ge C_{\varepsilon , j} > 0$  \ for all $j = 1 , ..., \nu$ and all  $x \in \overline{B}(x_{\varepsilon}, r_0)$.  The lemma follows by setting 
 $$C = \inf\{ \  C_{\varepsilon , j} \ : \    \varepsilon \in E,  j = 1 , ..., \nu  \} > 0 .$$
 \end{proof}

 \begin{proof}[Proof of Theorem 1.1]

Let   $u$ be solution of (1.11) defined on a maximal interval $[0, T)$, and denote by  $u(t)$ the smooth function on $M$ defined by $u(t)(x) = u(t,x) \  \forall x \in M$. Let $  \{ \varphi_1,..,\varphi_{\nu} \}$  be a basis of  $ \mathcal { N(Q)}$ satisfying hypothesis $(1.16)$.  Since $k_p = \int_M Q_0 dV_0 \not= 0$, then $\{ 1 ,  \varphi_1,..,\varphi_{\nu} \}$ is a basis of  
$ \mathcal{N}(P_0 )$. We denote by $\widehat{u}$ the orthogonal projection of $u$ onto $ \mathcal{N}(P_0 )$. We have 
$$
\widehat{u}(t,x) =  a_0 (t) +  \sum_{j=1}^{\nu} a_j(t) \varphi_j(x), \eqno (2.2)$$
for some $a_0(t), a_1(t), ..., a_{\nu}(t) \in \Ree.$

\medskip

\noindent In what follows $C_0$ denotes a positive constant depending on $(M, g_0)$, $f $ and the initial map $u_0$.
Since the functional  $II_f$ is decreasing along the flow (see 1.13), then \ $ II_f(u(t)) \le II_f(u_0)$ \ for all $t \in [0 , T)$. So 
$${n \over 2} \int_{M} P_0 u(t)\cdot u(t) \ dV_0 + n \int_M Q_0 u(t) \ dV_0 - k_p \log\left(\int_M f e^{nu(t)} dV_0\right)  \le C_0 ,$$
and since the volume of $M$ is constant along the flow (see (1.12) ), then 
$${1 \over 2} \int_{M} P_0 u(t)\cdot u(t) \ dV_0 +  \int_M Q_0 u(t) \ dV_0  \le C_0 . \eqno (2.3) $$
But we have by  (2.2), since $\varphi_j \in  \mathcal {N(Q)}$, 
$$\int_{M} Q_0 \widehat {u}(t)\ dV_0 = k_p a_0(t), $$
so it follows from (2.3) that 
$${1 \over 2}\int_{M} P_0 u(t) \cdot u(t)  \ dV_0  +  k_p a_0(t) \le  - \int_{M} Q_0 (u(t) - \widehat {u}(t)) \ dV_0 + C_0 . \eqno (2.4)$$
By Young's inequality and Poincar\'e's inequality, we have for all $ \varepsilon > 0$, 
$$
\left| \int_{M} Q_0 (u(t) - \widehat u(t)) \ dV_0  \right\vert \leq {1 \over  \varepsilon }   \int_{M} Q_0^2 \  dV_0 +  \varepsilon  \int_{M} (u(t)- \widehat u(t))^2 \  dV_0 $$

 $$
 \leq {1 \over  \varepsilon }   \int_{M} Q_0^2 \  dV_0 + { \varepsilon  \over \lambda_1}\int_{M} P_0 u(t) \cdot u(t)   \  dV_0 , \eqno (2.5)$$
where $ \lambda_1$ is the first positive eigenvalue of $P_0$. Choosing $  \varepsilon  = { \lambda_1 \over 4}$ and putting (2.5) into (2.4), we get
$$
{1 \over 4}\int_{M} P_0 u(t)\cdot u(t) \ dV_0  + k_p a_0(t)  \le  C_0 . \eqno (2.6)$$
On the other hand, we have for all $\delta \in (0 , n)$ (to be chosen later),
$$\int_M e^{\delta(\widehat u(t) -a_0(t))} dV_0  =  e^{-\delta a_0(t)} \int_M e^{\delta(\widehat{u}(t) -u(t)) + \delta u(t)} dV_0 ,$$
which gives by using H\"older's inequality 
$$\int_M e^{\delta(\widehat u(t) -a_0(t))} dV_0   \le  e^{-\delta a_0(t)}\left(\int_M e^{{\delta n\over n-\delta}(\widehat{u}(t) -u(t))} dV_0 \right)^{n-\delta \over n}
\left(\int_M e^{n u(t)} dV_0 \right)^{\delta  \over n}  $$
$$\le  C_0 e^{-\delta a_0(t)}\left(\int_M e^{{\delta n\over n-\delta}(\widehat{u}(t) -u(t))} dV_0 \right)^{n-\delta \over n} \eqno (2.7)$$
since the volume is constant along the flow (see (1.12)).  By Adams inequality (1.15) applied to $v= - {\delta \over n-\delta} u$, we have 
$$ \left(\int_M e^{{\delta n\over n-\delta}(\widehat{u}(t) -u(t))} dV_0 \right)^{n-\delta \over n}  \le  e^{C {n-\delta \over n } + C {\delta^2 \over n( n-\delta)}\int_M P_0u(t)\cdot u(t) dV_0 }$$
where $C$ is a positive constant depending only on $(M, g_0)$. Hence it follows from (2.7)  that 
$$ \int_M e^{\delta(\widehat u(t) -a_0(t))} dV_0     \le  C_0 e^{ -\delta a_0(t) } e^{ C {n-\delta \over n } + C {\delta^2 \over n( n-\delta)} \int_M P_0u(t)\cdot u(t) dV_0 }. \eqno (2.8) $$

 \noindent By (2.6) we have 
$$-\delta a_0(t) \le {\delta\over 4 k_p}\int_{M} P_0 u(t)\cdot u(t) \ dV_0   -  {\delta \over k_p} C_0, $$
which together with (2.8) give 
$$ \int_M e^{\delta(\widehat u(t) -a_0(t))} dV_0     \le C_0 e^{C {n-\delta \over n }- C_0 {\delta \over k_p}} e^{  \delta\left( {1 \over 4k_p} +  {C \delta \over n( n-\delta)}\right)\int_M P_0u(t)\cdot u(t) dV_0 }. \eqno (2.9) $$
Now since $k_p <0$, by choosing  $\delta \in (0, n) $  small enough such that  $  {1 \over 4k_p} +  {C \delta \over n( n-\delta)} \le 0$ (for example by taking $\delta = \min({-1\over C k_p}, 1)$), we have from  from (2.9), since $P_0$ is positive, that 
$$ \int_M e^{\delta(\widehat u(t) -a_0(t))} dV_0     \le  C_0 .  \eqno (2.10)$$
(Here $C_0$ is a new positive constant depending on $u_0, f $ and $(M, g_0)$.)

\medskip

\noindent We recall that $ \widehat u(t) -a_0(t)  = \sum_{j=1}^{\nu} a_j(t)\varphi_j$. If we let $\varepsilon(t)= (\varepsilon_1(t), ..., \varepsilon_{\nu}(t)):= \big( \hbox{sign}(a_1(t)), ...,  \hbox{sign}(a_{\nu}(t))\big) $, then it follows from Lemma 2.1 that  there exist  two positive constants  $r_0 , C$ (which do not depend on $t$) and a point  $x_{\varepsilon(t)} \in M $ such that, for all $ j = 1, ..., \nu$, we have 
$$\varepsilon_j(t) \varphi_j (x)  \ge C  \  \    \forall x \in B(x_{\varepsilon(t)}, r_0) ,$$
which gives  
$$  \sum_{j=1}^{\nu} a_j(t) \varphi_j(x)   \ge C \sum_{j=1}^{\nu}\left|a_j(t)\right|  \  \ \forall x  \in  B(x_{\varepsilon(t)}, r_0) .  \eqno (2.11)$$
Hence by combining (2.10) and (2.11) we obtain 
$$\int_{B(x_{\varepsilon(t)}, r_0)} e^{ \delta C \sum_{j=1}^{\nu}\left|a_j(t)\right| } dV_0 \le C_0$$
that is,  
$$ \left| B(x_{\varepsilon(t)}, r_0)\right| e^{\delta C \sum_{j=1}^{\nu}\left|a_j(t)\right| } \le C_0 ,  \eqno (2.12) $$
where $  \left| B(x_{\varepsilon(t)}, r_0)\right| $ is the volume of $ B(x_{\varepsilon(t)}, r_0) $ with respect to $g_0$.  But since $M$ is compact, then 
$  \left| B(x_{\varepsilon(t)}, r_0)\right| \ge C_1 r_0^{n}$, where the constant $C_1$ depends only on $(M, g_0)$.   It follows then from (2.12)  that 
$$ \sum_{j=1}^{\nu}\left|a_j(t)\right|  \le C_0 . \eqno (2.13) $$
Now, since $\widehat{u}$ is the orthogonal projection of $u$ onto $\mathcal{N}(P_0)$, then 
$\int_M( u-\widehat{u} ) dV_0 = 0 $.  This implies that 
$$\hbox{Vol}_{g_0}(M) a_0(t) = \int_M a_0(t) dV_0 = \int_M u(t) dV_0  - \sum_{j=1}^{\nu} a_j(t) \int_M \varphi_j dV_0 . \eqno (2.14)$$
We have,  by using the elementary inequality \ $ z \le e^z  \  \forall z \in \Ree$, 
$$\int_M u(t) \ dV_0 \le {1\over n} \int_M e^{nu(t)} dV_0  = {1\over n}\int_M e^{nu_0} dV_0  \eqno (2.15)$$
where we have used (1.12). 
Combining (2.13), (2.14) and (2.15) we obtain
$$ a_0(t) \le C_0 . \eqno (2.16) $$
Using the positivity of $P_0$ and the fact that  $k_p < 0$, it follows from (2.6) and (2.16) that 
 $$ |a_0(t)|  \le  C_0 . \eqno (2.17) $$
 Using again (2.6) we get also 
  $$\int_{M} P_0 u(t)\cdot u(t) \ dV_0  \le  C_0 . \eqno (2.18)$$

\noindent Now since $\mathcal { N(Q)}$ is of finite dimension,  it is clear from (2.2), (2.13) and (2.17) that $ \| \widehat u(t) \|_{H^{n/2}} \le C_0$, and since 
$$ \| u(t) -  \widehat u(t) \|_{H^{n/2}} \le C \int_{M} P_0 u(t)\cdot u(t) \ dV_0,$$
it follows from (2.18) that $u(t)$ is bounded in $H^{n\over 2}(M)$ uniformly in $t$.  Following the same arguments as in S. Brendle\cite{sB1}, one can prove that  $u(t)$ is bounded in $H^{k}(M)$ for all $k \ge n/2$ uniformly in time, so the evolution equation  (1.10)-(1.11) is defined for all time and moreover the solution converges to a metric  
 $g_{\infty}  = e^{2u_{\infty}}g_0$  with Q-curvature ${ k_p  \over \int_M f e^{2u_{\infty}} dV_0 } f$. 

\end{proof}

  \section{The positive case }
  
  \medskip

This section is devoted to the proof of Theorem 1.2 .

\begin{proof}[Proof of Theorem 1.2] Suppose that $u : [0,T) \times M \to \Ree $ with $u(0) = u_0$ is a local solution of (1.11)  where we suppose $f\equiv 1$ for simplicity, and denote by $u(t)$ the smooth function on $M$ defined by $u(t)(x) = u(t,x) \  \forall x \in M$.
Let $ \varphi \in \mathcal {N(Q)}$ such that 
$$\int_M \varphi e^{nu_0} dV_0 \not= 0.$$
 If  we multiply (1.11) by $ \varphi e^{nu(t) }$ and integrate over $M$, we get
$$
{ d \over  dt}  \int_M  \varphi  \ e^{nu(t)}\ dV_0  =  {n k_p \over 2 \int_M e^{nu(t)} dV_0 }  \int_M  \varphi  e^{nu(t)} \ dV_0 \  \ \forall t \in [0, T) .  \eqno (3.1)$$
But  the volume of $M$ is constant along the flow by (1.12),  that is,
$$ \int_M e^{nu(t)} dV_0 =   \int_M e^{nu_0} dV_0 .$$  
So   equation (3.1) becomes 
$$
{ d \over  dt}  \int_M  \varphi  \ e^{nu(t)}\ dV_0  =  {n k_p \over 2 \int_M e^{nu_0} dV_0 }  \int_M  \varphi  e^{nu(t)} \ dV_0 \ .  $$
Integrating this differential equation we get 
$$\int_M  \varphi  \ e^{nu(t)}\ dV_0 = e^{{nk_p \over 2 \int_M e^{nu_0} dV_0 } t} \int_M  \varphi  e^{nu_0} \ dV_0.$$
It follows that 
$$ e^{{nk_p \over 2 \int_M e^{nu_0} dV_0 } t}\left| \int_M  \varphi  e^{nu_0} \ dV_0\right| =\left| \int_M  \varphi  \ e^{nu(t)}\ dV_0\right|  \le \|\varphi\|_{L^{\infty}(M)}  \int_M  e^{nu_0} \ dV_0,$$
hence for any $t \in [0, T)$, we have 
$$ e^{{nk_p \over 2 \int_M e^{nu_0} dV_0 } t} \le { \|\varphi\|_{L^{\infty}(M)}  \int_M  e^{nu_0} \ dV_0 \over \left| \int_M  \varphi  e^{nu_0} \ dV_0\right| }.$$
This shows that the maximal time of existence $T$ must satisfy 
$$T \leq {2  \int_M e^{n u_0} \ dV_0 \over n k_p  } \log \left( {\|\varphi \|_{L^{\infty}(M)}  \int_M e^{n u_0} \ dV_0 \over \left|\int_M e^{n u_0} \phi \ dV_0 \right|}\right) .$$
This ends the proof of Theorem 1.2.
\end{proof}

 \bigskip

\begin{proof}[Proof of Remark 1.2] Let $\varphi \in \mathcal{N(Q)}$. We want to show that if $ v \in C^{\infty}(M)$ then  there exists a sequence $ (v_k)$  in $ C^{\infty} (M)$ such that $  \int_M \varphi e^{n v_k} \ dV_0 \not = 0  \ \forall k \in \Nee$ and \ $ (v_k)$ converges to $v$  in ${C^{\infty}(M)}$.  We  may suppose that $  \int_M e^{n v}   \varphi \ dV_0 = 0$ (else take $ v_k = v \ \forall k \in \Nee$), this implies in particular that  $\varphi$ is not of constant sign. Set $A^+$ = $\{x \in M : \varphi(x) > 0 \}$ and $A^- = \{x \in M : \varphi(x) < 0 \}$ and
let $h \in C^{\infty}(M) $ be a nonnegative function with support in $A^+$. Choose $\delta $  a positive real number sufficiently small such that $ \int_{A^+} h   \varphi \ dV_0 + \delta \int_{A^-}    \varphi \ dV_0 > 0$. So it follows that
$$\int_M (h+ \delta) \varphi \ dV_0 = \int_{A^+} (h+ \delta)  \varphi \ dV_0 + \int_{A^-} (h+ \delta ) \varphi \ dV_0 \ge \int_{A^+} h  \varphi \ dV_0 + \delta \int_{A^-}   \varphi \ dV_0 > 0 .$$
Hence if we set $w= {1  \over n} \log (h+ \delta)$, we have $  \int_M e^{n w}   \varphi \ dV_0 > 0$. It suffices to take  $ v_k = {1 \over n}  \log ( e^{nv} + {1 \over k} e^{nw} )$.
This ends the proof of Remark 1.2.

\end{proof}

\end{document}